\documentclass[11pt]{amsart}
\usepackage[margin=1.1in]{geometry}
\usepackage[T1]{fontenc}
\usepackage[utf8]{inputenc}
\usepackage{lmodern}
\usepackage{microtype}
\usepackage{mathtools,amssymb,amsthm}
\usepackage{amsmath}
\newcommand{\C}{\mathbb{C}}
\usepackage[hidelinks]{hyperref}

\newtheorem{theorem}{Theorem}[section]
\newtheorem{lemma}[theorem]{Lemma}
\newtheorem{corollary}[theorem]{Corollary}
\newtheorem{proposition}[theorem]{Proposition}
\theoremstyle{remark}
\newtheorem{remark}[theorem]{Remark}
\newtheorem{example}[theorem]{Example}

\newcommand{\Q}{\mathbf{Q}}
\newcommand{\Z}{\mathbf{Z}}
\newcommand{\F}{\mathbf{F}}

\title[Zero coefficients and rational Lambert series]{Zero coefficients of rational power series and rational Lambert series}
\author{Igor Rivin}
\address{Mathematics Department, Temple University and Dimension Reducers, LLC}
\email{igor@dimensionreducers.com}
\thanks{The question on Lambert Series was asked by an as-yet-unnamed OpenAI model, and this paper was written with the help of GPT5.4}
\date{March 2026}

\begin{document}
\begin{abstract}
We record several consequences of the Skolem-Mahler-Lech theorem in the language of rational power series and connect them to a rigidity theorem for Lambert series. First, the zero set of the coefficient sequence of a rational function is eventually periodic. Second, if all sufficiently large prime-index coefficients vanish, then after subtracting a polynomial the function is a finite sum of rational functions in proper powers $z^d$, although in general not a single such function. Third, if a rational power series has infinitely many zero coefficients, then among the poles of minimum modulus there must be a root-of-unity relation. We then prove the following Lambert-series rigidity statement: if $\gamma=(\gamma_n)_{n\ge 1}$ is eventually linearly recurrent and
\begin{equation*}
L_\gamma(z):=\sum_{n\ge 1}\gamma_n\frac{z^n}{1-z^n}
\end{equation*}
is rational, then $\gamma$ is finitely supported. The proof uses the same finite-place setup as Skolem-Mahler-Lech, but in the Lambert setting the divisor-sum identity produces the contradiction before any $p$-adic interpolation step is needed.
\end{abstract}
\maketitle
\section{Introduction}
Let
\[
F(z)=\sum_{n\ge 0} a_n z^n\in \C[[z]].
\]
When $F$ is rational, the coefficient sequence $(a_n)$ satisfies a constant-coefficient linear recurrence from some point on, and conversely every eventually linearly recurrent sequence has rational generating function. The classical Skolem-Mahler-Lech theorem says that the zero set
\[
Z(F):=\{n\ge 0:a_n=0\}
\]
is eventually periodic: it is a finite set together with finitely many arithmetic progressions. We use this in two elementary ways. First, if all sufficiently large prime-index coefficients vanish, then the eventual support of $(a_n)$ can only occupy residue classes that are not coprime to the ambient modulus; this yields a decomposition of $F$ into a finite sum of rational functions in proper powers $z^d$. Second, infinitely many zero coefficients force a root-of-unity relation among the dominant characteristic roots, or equivalently among the poles of minimum modulus.

The questions considered here grew out of Lambert-series identities in Rivin's work on free groups \cite{Rivin} and from the rigidity problem for the operator
\[
(T_\gamma G)(z):=\sum_{n\ge 1}\gamma_n G(z^n).
\]
Given a sequence $\gamma=(\gamma_n)_{n\ge 1}$, define its Lambert series by
\[
L_\gamma(z):=\sum_{n\ge 1}\gamma_n\frac{z^n}{1-z^n}.
\]
Expanding each summand geometrically gives
\begin{equation}\label{eq:lambert-coeffs}
L_\gamma(z)=\sum_{m\ge 1} b_m z^m,
\qquad
b_m:=\sum_{d\mid m}\gamma_d.
\end{equation}
Thus the coefficients of $L_\gamma$ are the divisor sums of $\gamma$.

Our main theorem says that the recurrence world is rigid.

\begin{theorem}\label{thm:main-arith}
Let $\gamma=(\gamma_n)_{n\ge 1}\subset \C$, and put
\[
b_n:=\sum_{d\mid n}\gamma_d \qquad (n\ge 1).
\]
If both $\gamma$ and $b=(b_n)_{n\ge 1}$ are eventually linearly recurrent, then $\gamma$ is finitely supported.
\end{theorem}

The main generating-function consequence is immediate.

\begin{corollary}\label{cor:main-gf}
Let $\gamma=(\gamma_n)_{n\ge 1}\subset \C$, and write
\[
\Gamma_\gamma(z):=\sum_{n\ge 1}\gamma_n z^n,
\qquad
L_\gamma(z):=\sum_{n\ge 1}\gamma_n\frac{z^n}{1-z^n}.
\]
If both $\Gamma_\gamma(z)$ and $L_\gamma(z)$ are rational, then $\Gamma_\gamma(z)$ is a polynomial. Equivalently, among sequences with rational ordinary generating function, the only ones whose Lambert series is rational are the finitely supported sequences.
\end{corollary}

For classical background on Lambert series see Hardy and Wright \cite{HardyWright} and the survey of Schmidt \cite{Schmidt}. For linear recurrence sequences see Everest, van der Poorten, Shparlinski, and Ward \cite{Everest}. For the Skolem-Mahler-Lech theorem and its variants see Lech \cite{Lech}, Bell \cite{Bell}, Adamczewski and Bell \cite{AdamczewskiBell}, and van der Poorten and Schlickewei \cite{PoortenSchlickewei}.

\section{Linear recurrences and the Skolem-Mahler-Lech viewpoint}

We say that a sequence $u=(u_n)_{n\ge 0}$ is \emph{eventually linearly recurrent} if there exist $r\ge 1$, coefficients $c_1,\dots,c_r\in \C$ with $c_r\neq 0$, and an index $N$ such that
\begin{equation}\label{eq:recurrence-start}
u_{n+r}=c_1u_{n+r-1}+\cdots+c_r u_n \qquad (n\ge N).
\end{equation}
If one may take $N=0$, we say that the recurrence holds \emph{from the start}.

\begin{lemma}\label{lem:rational-recurrence}
For a complex sequence $(u_n)_{n\ge 0}$, the following are equivalent.
\begin{enumerate}
\item The generating function $\sum_{n\ge 0} u_n z^n$ is rational.
\item The sequence $(u_n)$ is eventually linearly recurrent.
\end{enumerate}
Moreover, if $(u_n)$ is eventually linearly recurrent, then one can write
\[
u=v+f,
\]
where $f$ is finitely supported and $v$ is linearly recurrent from the start.
\end{lemma}

\begin{proof}
The equivalence is classical; see, for example, \cite[Chapter~3]{Everest}. For the final assertion, choose $N,r$ and $c_1,\dots,c_r$ with $c_r\neq 0$ such that \eqref{eq:recurrence-start} holds for all $n\ge N$. Define $v_n:=u_n$ for $n\ge N$. Since $c_r\neq 0$, the recurrence can be solved backward:
\[
v_n=c_r^{-1}\bigl(v_{n+r}-c_1v_{n+r-1}-\cdots-c_{r-1}v_{n+1}\bigr)
\qquad (0\le n<N).
\]
Then $v$ satisfies the same recurrence for every $n\ge 0$, and $f:=u-v$ is finitely supported.
\end{proof}

\begin{theorem}[Skolem-Mahler-Lech for rational power series]\label{thm:SML-rational}
Let
\[
F(z)=\sum_{n\ge 0} a_n z^n\in \C(z).
\]
Then the zero set
\[
Z(F):=\{n\ge 0:a_n=0\}
\]
is a finite union of arithmetic progressions together with a finite set. In particular, $Z(F)$ is eventually periodic.
\end{theorem}

\begin{proof}[Proof sketch]
By Lemma \ref{lem:rational-recurrence}, after subtracting a polynomial we may assume that $(a_n)$ satisfies a recurrence from the start, so that
\[
a_n=\ell(A^n v)
\qquad (n\ge 0)
\]
for some matrix $A\in \operatorname{GL}_r(K)$, vector $v\in K^r$, and linear form $\ell$, where $K\subset \C$ is a finitely generated field over $\Q$ containing all the data.

Choose a good nonarchimedean place of $K$ at which the entries of $A$, $A^{-1}$, $v$, and $\ell$ are integral. Passing to the corresponding completion, let $\pi$ be a uniformizer. Since the residue field is finite and $\bar A$ is invertible, some power $A^M$ satisfies
\[
A^M=I+\pi C
\]
with $C$ integral.

For each residue class $j\bmod M$, the subsequence $a_{Mn+j}$ is interpolated by a $p$-adic analytic function on $\Z_p$:
\[
f_j(x):=\ell\bigl(A^j(A^M)^x v\bigr),
\qquad
(A^M)^x:=\sum_{m\ge 0}\binom{x}{m}(A^M-I)^m.
\]
The binomial series converges because $A^M-I$ is $\pi$-adically small, and for every integer $n\ge 0$ one has $f_j(n)=a_{Mn+j}$.

If $a_{Mn+j}=0$ for infinitely many $n$, then $f_j$ has infinitely many zeros in $\Z_p$, hence vanishes identically by Strassmann's theorem. Therefore the entire progression $Mn+j$ belongs to $Z(F)$. Since there are only finitely many residue classes modulo $M$, the zero set is a finite union of arithmetic progressions up to finitely many initial exceptions. See \cite{Lech}, \cite{Bell}, or \cite[\S1]{AdamczewskiBell} for fuller accounts.
\end{proof}

\begin{remark}\label{rem:SML-nondegenerate}
A standard refinement says that if the recurrence is \emph{nondegenerate}, meaning that no quotient of two distinct characteristic roots is a root of unity, then the sequence has only finitely many zeros; see, for instance, \cite[Chapter~3]{Everest} or \cite{PoortenSchlickewei}. The arithmetic progressions in Theorem \ref{thm:SML-rational} arise precisely from root-of-unity relations among characteristic roots.
\end{remark}

\section{Two consequences for rational power series}

We first record the structural consequence for prime-index coefficients.

\begin{proposition}\label{prop:prime-support}
Let
\[
F(z)=\sum_{n\ge 0} a_n z^n\in \C(z),
\]
and assume that $a_p=0$ for all sufficiently large primes $p$. Then there exist a polynomial $P(z)$, integers $d_1,\dots,d_m>1$, and rational functions $H_1,\dots,H_m\in \C(z)$ such that
\[
F(z)=P(z)+\sum_{j=1}^m H_j(z^{d_j}).
\]
In general one cannot take a single $d>1$.
\end{proposition}

\begin{proof}
By Theorem \ref{thm:SML-rational}, the support
\[
S:=\{n\ge 0:a_n\neq 0\}
\]
is eventually periodic. Thus there exist $k\ge 1$, a finite set $E\subset \Z_{\ge 0}$, and residues $r_1,\dots,r_t\in \{0,1,\dots,k-1\}$ such that for all sufficiently large $n$,
\[
n\in S \iff n\equiv r_j \pmod{k}\ \text{for some }j.
\]
If one of the residues $r_j$ is coprime to $k$, then Dirichlet's theorem gives infinitely many primes $p\equiv r_j\pmod{k}$ \cite[Theorem~15]{HardyWright}. For all sufficiently large such primes we would have $a_p\neq 0$, contradicting the hypothesis. Hence every eventual support residue class satisfies
\[
\gcd(r_j,k)>1.
\]

Let $P(z)$ collect the finitely many terms with indices in $E$ and those below the eventual periodic range. For each residue class $r\bmod k$, define
\[
F_r(z):=\sum_{n\equiv r\,({\rm mod}\, k)} a_n z^n.
\]
By the roots-of-unity filter,
\[
F_r(z)=\frac1k\sum_{j=0}^{k-1}\omega^{-jr}F(\omega^j z),
\qquad \omega=e^{2\pi i/k},
\]
so each $F_r$ is rational. Writing
\[
F_r(z)=z^r G_r(z^k)
\]
for a rational function $G_r$, we obtain
\[
F(z)=P(z)+\sum_{\gcd(r,k)>1} z^r G_r(z^k).
\]
If $d_r:=\gcd(r,k)>1$, then $r=d_r r'$ and $k=d_r k'$, so
\[
z^r G_r(z^k)=(z^{d_r})^{r'}G_r\bigl((z^{d_r})^{k'}\bigr),
\]
which is a rational function of $z^{d_r}$. This yields the required decomposition.
\end{proof}

\begin{example}\label{ex:single-d-false}
The single-$d$ version is false. Consider
\[
F(z)=\frac{z^4}{1-z^2}+\frac{z^9}{1-z^3}.
\]
Its coefficients vanish at every prime index: $a_2=a_3=0$, and for every prime $p>3$ one has $p\not\equiv 0\pmod 2$ and $p\not\equiv 0\pmod 3$, hence $a_p=0$. But $F$ is not of the form $P(z)+G(z^d)$ for a single $d>1$, because its support contains infinitely many powers of $2$ and infinitely many powers of $3$.
\end{example}

We next isolate the dominant-pole obstruction to infinitely many zeros.

\begin{proposition}\label{prop:dominant-poles}
Let
\[
F(z)=\sum_{n\ge 0} a_n z^n\in \C(z),
\]
and suppose that $a_n=0$ for infinitely many $n$. Then among the poles of $F$ having minimum modulus, at least two have quotient a root of unity.
\end{proposition}

\begin{proof}
By Theorem \ref{thm:SML-rational}, there exist integers $M\ge 1$ and $r\ge 0$ such that
\[
a_{Mn+r}=0
\qquad (n\ge 0)
\]
after discarding finitely many initial terms. Write the coefficient sequence in exponential-polynomial form
\[
a_n=\sum_{i=1}^s P_i(n)\lambda_i^n,
\]
where the $\lambda_i\in \C^\times$ are distinct and the $P_i\in \C[n]$ are nonzero. Let
\[
\rho:=\max_i |\lambda_i|.
\]
The poles of minimum modulus are precisely the reciprocals of those $\lambda_i$ with modulus $\rho$.

Expanding along the progression $Mn+r$ gives
\[
0=a_{Mn+r}=\sum_{i=1}^s Q_i(n)(\lambda_i^M)^n,
\]
where $Q_i(n)=\lambda_i^r P_i(Mn+r)$ is a nonzero polynomial. Suppose that among the dominant roots $|\lambda_i|=\rho$, the numbers $\lambda_i^M$ are pairwise distinct. Grouping together equal bases among the $\lambda_i^M$, we obtain an identity
\[
\sum_{\mu} R_\mu(n)\mu^n=0
\qquad (n\ge 0),
\]
with distinct nonzero bases $\mu$ and polynomial coefficients $R_\mu$. Such an exponential-polynomial representation is unique; equivalently, the corresponding partial-fraction expansion of the generating function is unique. Hence every $R_\mu$ must vanish identically. In particular the dominant terms vanish, contradicting the nonzero polynomials $Q_i$ attached to them. Therefore two dominant roots satisfy
\[
\lambda_i^M=\lambda_j^M,
\]
so $\lambda_i/\lambda_j$ is a root of unity. Passing to reciprocals gives the corresponding statement for poles.
\end{proof}

\begin{example}\label{ex:irrational-poles}
Irrationally related dominant poles may coexist with infinitely many zeros, but they are not the mechanism producing those zeros. Fix $\theta/\pi\notin \Q$ and set
\[
a_n=(1-(-1)^n)(1+e^{in\theta}).
\]
Then $a_{2m}=0$ for every $m$, while
\[
\sum_{n\ge 0} a_n z^n
=\frac{1}{1-z}+\frac{1}{1-e^{i\theta}z}-\frac{1}{1+z}-\frac{1}{1+e^{i\theta}z}.
\]
The four poles all have modulus $1$, and the ratio between $1$ and $e^{-i\theta}$ is not a root of unity. Nevertheless, there are root-of-unity pairs $1,-1$ and $e^{-i\theta},-e^{-i\theta}$, and these are what force the vanishing on the even progression.
\end{example}

\section{Finite specialization and periodicity over finite fields}

The proof of the Lambert-series theorem uses the same finite-place setup as the proof sketch of Theorem \ref{thm:SML-rational}, but it only needs reduction to a finite residue field.

\begin{lemma}[finite specialization]\label{lem:finite-specialization}
Let $A\subset \C$ be a finitely generated domain over $\Z$, and let $s\in A\setminus\{0\}$. Then there exist a finite field $k$ and a ring homomorphism
\[
\varphi:A[s^{-1}]\longrightarrow k.
\]
In particular, one may reduce any finite amount of data from $A[s^{-1}]$ modulo a finite field while keeping the image of $s$ nonzero.
\end{lemma}

\begin{proof}
The localization $A[s^{-1}]$ is a nonzero finitely generated $\Z$-algebra, so it has a maximal ideal $\mathfrak m$. Then $k:=A[s^{-1}]/\mathfrak m$ is a field finitely generated as a $\Z$-algebra. Its characteristic is some prime $p$, and $k$ is therefore a finitely generated algebra over $\F_p$. Since $k$ is also a field, Zariski's lemma implies that $k$ is a finite extension of $\F_p$, hence a finite field; see, for example, Atiyah--Macdonald \cite{AtiyahMacdonald} or Eisenbud \cite{Eisenbud}.
\end{proof}

\begin{lemma}[periodicity over finite fields]\label{lem:periodicity-finite-field}
Let $k$ be a finite field.
\begin{enumerate}
\item If a sequence $(u_n)_{n\ge 0}\subset k$ satisfies a recurrence
\[
u_{n+r}=a_1u_{n+r-1}+\cdots+a_r u_n \qquad (n\ge 0)
\]
with $a_r\neq 0$, then $(u_n)$ is periodic.
\item If a sequence $(w_n)_{n\ge 0}\subset k$ satisfies a recurrence
\[
w_{n+s}=d_1w_{n+s-1}+\cdots+d_s w_n \qquad (n\ge N)
\]
for some $N$, then $(w_n)$ is eventually periodic.
\end{enumerate}
\end{lemma}

\begin{proof}
For the first claim, set
\[
U_n:=(u_n,u_{n+1},\dots,u_{n+r-1})\in k^r.
\]
Then $U_{n+1}=MU_n$, where $M$ is the companion matrix of the recurrence. Since $\det M=(-1)^{r+1}a_r\neq 0$, we have $M\in \mathrm{GL}_r(k)$. The group $\mathrm{GL}_r(k)$ is finite, so $M^T=I$ for some $T\ge 1$, and therefore $U_{n+T}=U_n$ for every $n\ge 0$.

For the second claim, form the state vectors
\[
W_n:=(w_n,w_{n+1},\dots,w_{n+s-1})\in k^s \qquad (n\ge N).
\]
These evolve by a deterministic map $W_{n+1}=\Psi(W_n)$ on the finite set $k^s$. Hence the orbit of $W_N$ is eventually periodic, and so is $(w_n)$.
\end{proof}

\section{Proof of the Lambert-series rigidity theorem}

\begin{proof}[Proof of Theorem \ref{thm:main-arith}]
Assume that $\gamma$ and $b$ are eventually linearly recurrent, where
\[
b_n=\sum_{d\mid n}\gamma_d.
\]
We must show that $\gamma$ is finitely supported.

By Lemma \ref{lem:rational-recurrence}, write $\gamma=\eta+f$ with $f$ finitely supported and $\eta$ linearly recurrent from the start. Then
\[
\sum_{d\mid n}\eta_d=b_n-\sum_{d\mid n}f_d.
\]
The second term on the right is periodic, because if $f_d=0$ for $d>N$ then
\[
\sum_{d\mid n}f_d=\sum_{d\le N} f_d\,\mathbf{1}_{d\mid n}.
\]
Hence the divisor-sum sequence of $\eta$ is still eventually linearly recurrent. Since $\gamma$ is finitely supported if and only if $\eta$ is, we may replace $\gamma$ by $\eta$ and assume from now on that
\begin{equation}\label{eq:gamma-recurrence}
\gamma_{n+r}=c_1\gamma_{n+r-1}+\cdots+c_r\gamma_n \qquad (n\ge 1)
\end{equation}
for some $r\ge 1$ and $c_r\neq 0$.

Suppose for contradiction that $\gamma$ is not finitely supported. Then $\gamma\neq 0$, so by M\"obius inversion there exists $m\ge 1$ such that
\begin{equation}\label{eq:S-def}
S:=\sum_{d\mid m}\gamma_d\neq 0.
\end{equation}
Indeed, if all such divisor sums were zero, then
\[
\gamma_n=\sum_{d\mid n}\mu(d)\sum_{e\mid n/d}\gamma_e=0
\qquad (n\ge 1).
\]

Let $N,s$ and $d_1,\dots,d_s\in \C$ be such that
\begin{equation}\label{eq:b-recurrence}
b_{n+s}=d_1b_{n+s-1}+\cdots+d_s b_n \qquad (n\ge N).
\end{equation}
Choose a finitely generated domain $A\subset \C$ over $\Z$ containing
\[
\{c_1,\dots,c_r,\gamma_1,\dots,\gamma_r, d_1,\dots,d_s, b_N,\dots,b_{N+s-1}, S\}.
\]
By \eqref{eq:gamma-recurrence} and \eqref{eq:b-recurrence}, every $\gamma_n$ and every $b_n$ lies in $A[c_r^{-1}]$. Applying Lemma \ref{lem:finite-specialization} to the localization
\[
R:=A[(c_rS)^{-1}],
\]
we obtain a finite field $k$ and a homomorphism
\[
\varphi:R\to k.
\]
Write bars for reduction via $\varphi$.

The reduced sequence $\bar\gamma$ satisfies the recurrence \eqref{eq:gamma-recurrence} over $k$ with trailing coefficient $\bar c_r\neq 0$. By Lemma \ref{lem:periodicity-finite-field}(1), $\bar\gamma$ is periodic; let $T_\gamma$ be a period. The reduced sequence $\bar b$ satisfies \eqref{eq:b-recurrence} for all $n\ge N$, so by Lemma \ref{lem:periodicity-finite-field}(2) it is eventually periodic; let $T_b$ be a period valid for all $n\ge N_0$.

Let $T$ be a common multiple of $T_\gamma$ and $T_b$. By Dirichlet's theorem, there exist infinitely many rational primes $q$ with $q\equiv 1\pmod T$; see, for example, \cite[Theorem~15]{HardyWright}. Choose one such prime satisfying also $q\nmid m$ and $mq\ge N_0$.

Since $mq^2\equiv mq\pmod{T_b}$ and both indices are at least $N_0$, eventual periodicity gives
\begin{equation}\label{eq:b-equal}
\bar b_{mq^2}=\bar b_{mq}.
\end{equation}
On the other hand, because $q\nmid m$,
\[
b_{mq^2}-b_{mq}
=\sum_{d\mid mq^2}\gamma_d-\sum_{d\mid mq}\gamma_d
=\sum_{e\mid m}\gamma_{eq^2}.
\]
Reducing modulo $k$ and using $q^2\equiv 1\pmod{T_\gamma}$, we obtain
\[
\bar b_{mq^2}-\bar b_{mq}
=\sum_{e\mid m}\bar\gamma_{eq^2}
=\sum_{e\mid m}\bar\gamma_e
=\overline{S}.
\]
But $S$ was inverted in $R$, so $\overline{S}\neq 0$ in $k$. This contradicts \eqref{eq:b-equal}. Therefore $\gamma$ must be finitely supported.
\end{proof}

\begin{proof}[Proof of Corollary \ref{cor:main-gf}]
A sequence has rational ordinary generating function if and only if it is eventually linearly recurrent; see \cite[Chapter~3]{Everest}. Apply Theorem \ref{thm:main-arith} to $\gamma$ and to the coefficient sequence $b$ of $L_\gamma$ given by \eqref{eq:lambert-coeffs}. If $\gamma$ is finitely supported, then both $\Gamma_\gamma(z)$ and $L_\gamma(z)$ are obviously rational, and $\Gamma_\gamma(z)$ is a polynomial.
\end{proof}

\begin{remark}\label{rem:lambert-vs-sml}
The proof of Theorem \ref{thm:main-arith} is very close in spirit to the proof sketch of Theorem \ref{thm:SML-rational}. In both arguments one reduces a recurrence to a finite field, obtains periodicity or eventual periodicity there, and then compares values along carefully chosen arithmetic patterns. The difference is that Skolem-Mahler-Lech passes from periodicity modulo a finite place to identically vanishing subsequences by means of $p$-adic interpolation and Strassmann's theorem. In the Lambert setting the special identity
\[
b_{mq^2}-b_{mq}=\sum_{e\mid m}\gamma_{eq^2}
\]
already produces a fixed nonzero quantity after reduction, so the contradiction is obtained before any $p$-adic analytic step is needed.
\end{remark}

\section{Further corollaries and examples}

We first record a partial consequence of the prime-index vanishing theorem. It is weaker than Theorem \ref{thm:main-arith}, but it already shows that prime and prime-square coefficients cannot behave independently.

\begin{proposition}[prime and prime-square coefficients]\label{prop:prime-primesquare}
Let $\gamma=(\gamma_n)_{n\ge 1}\subset \C$, and put
\[
b_n:=\sum_{d\mid n}\gamma_d \qquad (n\ge 1).
\]
Assume that both $\gamma$ and $b=(b_n)_{n\ge 1}$ are eventually linearly recurrent. Define
\[
c_n:=b_n-\gamma_n-\gamma_1 \qquad (n\ge 1).
\]
Then $c=(c_n)_{n\ge 1}$ is eventually linearly recurrent and satisfies
\[
c_p=0
\]
for every prime $p$. Consequently,
\[
\sum_{n\ge 1} c_n z^n=P(z)+\sum_{j=1}^m H_j(z^{d_j})
\]
for some polynomial $P(z)$, integers $d_1,\dots,d_m>1$, and rational functions $H_1,\dots,H_m\in \C(z)$. In particular,
\[
\gamma_p=0,
\qquad
b_p=\gamma_1
\]
for all sufficiently large primes $p$. If $\gamma_1=0$, then both $\gamma_p$ and $b_p$ vanish for all sufficiently large primes $p$.
\end{proposition}

\begin{proof}
Since $\gamma$ and $b$ are eventually linearly recurrent, so is $c$. For a prime $p$ we have
\[
c_p=b_p-\gamma_p-\gamma_1=(\gamma_1+\gamma_p)-\gamma_p-\gamma_1=0.
\]
Applying Proposition \ref{prop:prime-support} to the rational generating function $\sum_{n\ge 1} c_n z^n$ yields the stated decomposition.

By the proof of Proposition \ref{prop:prime-support}, there exist an integer $k\ge 1$ and a finite set of residue classes $R\subset \{0,1,\dots,k-1\}$ such that, for all sufficiently large $n$,
\[
c_n\neq 0 \implies n\equiv r \pmod k \text{ for some } r\in R,
\qquad
\gcd(r,k)>1.
\]
If $p$ is a sufficiently large prime with $p\nmid k$, then $p^2$ is coprime to $k$, so $p^2$ lies in none of the eventual support classes. Hence $c_{p^2}=0$. On the other hand,
\[
c_{p^2}=b_{p^2}-\gamma_{p^2}-\gamma_1
=(\gamma_1+\gamma_p+\gamma_{p^2})-\gamma_{p^2}-\gamma_1
=\gamma_p.
\]
Therefore $\gamma_p=0$ for all sufficiently large primes $p$, and then
\[
b_p=\gamma_1+\gamma_p=\gamma_1
\]
for all sufficiently large primes $p$.
\end{proof}

\begin{remark}
Proposition \ref{prop:prime-primesquare} is a useful first layer of Theorem \ref{thm:main-arith}: comparing prime and prime-square coefficients forces vanishing at prime indices. What it does \emph{not} do by itself is relate $\gamma_{eq^2}$ back to $\gamma_e$ for fixed $e$. That extra step is precisely what the finite-place argument in the proof of Theorem \ref{thm:main-arith} supplies after reduction modulo a suitable finite field.
\end{remark}

The original operator question can be phrased as follows. Given $\gamma=(\gamma_n)_{n\ge 1}$, define
\[
(T_\gamma G)(z):=\sum_{n\ge 1}\gamma_n G(z^n),
\]
whenever the series makes sense formally near $z=0$. Since
\[
T_\gamma\!\left(\frac{z}{1-z}\right)=L_\gamma(z),
\]
one test function already suffices in the recurrent setting.

\begin{corollary}\label{cor:operator}
Let $\gamma=(\gamma_n)_{n\ge 1}$ be eventually linearly recurrent. If
\[
T_\gamma\!\left(\frac{z}{1-z}\right)
=\sum_{n\ge 1}\gamma_n\frac{z^n}{1-z^n}
\]
is rational, then $\gamma$ is finitely supported. In particular, if $T_\gamma$ sends every rational function with poles at roots of unity to a rational function, then $\gamma$ is finitely supported.
\end{corollary}

\begin{corollary}[periodic coefficients]
If $\gamma$ is periodic and $L_\gamma(z)$ is rational, then $\gamma$ is identically zero.
\end{corollary}

\begin{proof}
A periodic sequence is linearly recurrent from the start. If it is finitely supported, it must be the zero sequence. Now apply Corollary \ref{cor:main-gf}.
\end{proof}

\begin{corollary}[Fibonacci example]
Let $(F_n)_{n\ge 1}$ be the Fibonacci sequence. Then
\[
\sum_{n\ge 1}F_n\frac{z^n}{1-z^n}
\]
is not a rational function.
\end{corollary}

\begin{proof}
The Fibonacci sequence is linearly recurrent and not finitely supported, so Corollary \ref{cor:main-gf} applies.
\end{proof}

\begin{remark}
The proof of Theorem \ref{thm:main-arith} is naturally nonarchimedean. One can read Lemma \ref{lem:finite-specialization} as the choice of a good finite place of the finitely generated field generated by the recurrence data. After reduction at that place, recurrent sequences become periodic or eventually periodic, and the prime-power comparison $mq$ versus $mq^2$ forces a fixed nonzero divisor sum to vanish. The argument does not require the coefficients of the recurrence to be algebraic; finite generation over $\Z$ is enough.
\end{remark}

\end{document}